\documentclass[10pt]{article}
\usepackage{mathrsfs}
\usepackage{amsfonts}
\usepackage{amsthm,amsmath,amssymb,anysize}
\usepackage[all]{xy}
\newtheorem{lemma}{Lemma}[section]
\newtheorem{theorem}[lemma]{Theorem}

\newtheorem{proposition}[lemma]{Proposition}
\newtheorem{definition}[lemma]{Definition}

\usepackage[numbers,sort&compress]{natbib}
\marginsize{3.5cm}{3.5cm}{3.6cm}{3cm}
\numberwithin{equation}{section}

\title{\textsf{Multipliers, Covers and Stem Extensions for  Lie   Superalgebras}}
\author{\textsc{Xingxue Miao and
  \textsc{Wende Liu}}\footnote{Correspondence:  wendeliu@ustc.edu.cn (W. Liu), supported by the NSF
  of China (11171055, 11471090)}\;\; \;
  \\
  \\
  \ \ \textit{School of Mathematical Sciences},
  \textit{Harbin Normal University} \\
  \textit{Harbin 150025, China}
  }
\date{ }

\begin{document}
\maketitle
\begin{quotation}
\small\noindent \textbf{Abstract}:
Suppose that the underlying field  is of characteristic different from $2, 3$. In this paper we first prove that   the so-called stem deformations of a free presentations of a finite-dimensional Lie superalgebra $L$  exhaust  all the maximal stem extensions of $L$, up to equivalence of extensions. Then we prove that  multipliers and   covers  always exist for a Lie superalgebra  and they are unique up to Lie superalgebra isomorphisms. Finally,  we describe the multipliers, covers and  maximal stem extensions of Heisenberg superalgebras of odd centers and model filiform Lie superalgebras.

\vspace{0.2cm} \noindent{\textbf{Keywords}}: multiplier, cover, stem extension, Heisenberg superalgebra, filiform Lie supleralgebra

\vspace{0.2cm} \noindent{\textbf{Mathematics Subject Classification 2010}}: 17B05, 17B30, 17B56
\end{quotation}

\setcounter{section}{0}
\section{Introduction}
The notion of multipliers and covers, which is closed relative to stem extensions,  first appeared Schur's work in the group theory and then was generalized to Lie algebra case. It proves that the theory of multipliers, covers and stem extension is not only of intrinsic interest, but also of important role in characterizing  algebraic structures, such as in computing second cohomology with coefficients in trivial modules for  groups or Lie algebras.

The study on multipliers  of Lie algebras began in 1990's (see \cite{Batten-Stitzinger,K.Moneyhun}, for example) and  the theory  has seen a fruitful  development (see \cite{B-M-S,E-S-D,Hardy,Hardy-stitzinger,Nir,Niroomand3,SAN}, for example). Among them,  a typical fact analogous to the one in the group theory  is that the multiplier of a finite-dimensional Lie algebra $L$ is isomorphic to the second cohomology group of $L$ with coefficients in the 1-dimensional trivial module (see \cite{Batten}, for example).

The notions of multipliers, covers and stem extensions may also be naturally generalized to Lie superalgebra case.
In this paper,     we first introduce the notions of stem denominators and stem deformations for an extension of a Lie superalgebra and show that all  the stem deformations of a free presentations of a finite-dimensional Lie superalgebra $L$ coincide all the maximal stem extensions of $L$, up to equivalence of extensions. Then we show that  multipliers and   covers  always exist for a Lie superalgebra  and they are unique up to Lie superalgebra isomorphisms. Finally, we describe multipliers, covers and   maximal stem extensions of Heisenberg superalgebras of odd centers and model filiform Lie superalgebras.

\section{Stem extensions}

 Unless otherwise stated, we assume that  the underlying field $\mathbb{F}$ is of characteristic different from $2, 3$, and all (super)spaces, (super)algebras are defined over $\mathbb{F}$.
 Let
  $\mathbb{Z}_{2}:=\{\bar{0}, \bar{1}\}$ be the additive group of order $2$ and $V=V_{\bar{0}}\oplus
V_{\bar{1}}$ a superspace, that is, a $\mathbb{Z}_2$-graded vector space.  For a  homogeneous element $x$ in $V$, write $|x|$ for the
parity of $x$. The symbol $|x|$  implies that $x$ has been assumed to be  a  homogeneous element.

Note that in a superspace, a subsuperspace always has a supplementary subsuperspace. Moreover, if $V$ is a superspace and $W$ is a subsuperspace of $V$, then the quotient space $V/W$ inherits a super structure and every subsuperspace of $V/W$ is of form $X/W$, where $X$ is a subsuperspace containing $W.$

Let us recall the notion of extensions Lie superalgebras and some basic properties. By definition a Lie superalgebra homomorphism is both an even linear map and an algebra homomorphism  and an ideal of a Lie superalgebra is always  a $\mathbb{Z}_{2}$-graded ideal. An extension of a Lie superalgebra $L$ by  $A$ is an exact sequence of Lie superalgebra homomorphisms:
\begin{equation}\label{eqstemext}
0\longrightarrow A\stackrel{\alpha}{\longrightarrow} B \stackrel{\beta}{\longrightarrow} L \longrightarrow 0.
\end{equation}
We usually identify $A$ with a subalgebra of $B$ and omit the embedding map $\alpha$.
Suppose that
\begin{equation}\label{eqstemextanthor}
0\longrightarrow C{\longrightarrow} D \stackrel{\gamma}{\longrightarrow} L \longrightarrow 0
\end{equation}
is also an extension of $L$ and there is a homomorphism $f$ from extensions (\ref{eqstemext}) to (\ref{eqstemextanthor}), that is,  a Lie superalgebra homomorphism $f: B\longrightarrow D$ such that $\gamma \circ f=\beta$. Then $f$ maps $A$ into $C$ and $f^{-1}(C)=A$. Hereafter we denote by $f$ itself the restriction $f: A\longrightarrow C$. Then the following diagram commutes:
$$
\xymatrix{0\ar[r]&A\ar[d]_{f}\ar[r]^{}&B\ar[d]_{f}\ar[r]^{\beta}&L\ar[d]_{\mathrm{id}}\ar[r]&0\\
0\ar[r]&C\ar[r]^{}&D\ar[r]^{\gamma}&L\ar[r]&0}
$$
 Clearly, if $f: B\longrightarrow D$ is surjective, then so is its restriction $\underline{f}: A\longrightarrow C$.

Recall that an extension (\ref{eqstemext}) is said to be central if the kernel $A$ is contained in the center  $\mathrm{Z}(B)$ of $B$. If $X$ is an ideal of $A$ as well as $B$, then (\ref{eqstemext}) induces an exact sequence
\begin{equation*}\label{eqstemextxyz}
0\longrightarrow A/X{\longrightarrow} B/X \stackrel{\beta}{\longrightarrow} L \longrightarrow 0.
\end{equation*}
Hereafter, we denote by $\beta$ itself the map induced $\overline{\beta}: B/X\longrightarrow L$.

As in the Lie algebra case, we introduce the following notion.
\begin{definition}\label{defstem} A stem extension of a Lie superalgebra $L$ is a central extension
\begin{equation*}
0\longrightarrow S{\longrightarrow} T \stackrel{\beta}{\longrightarrow} L \longrightarrow 0
\end{equation*}
such that
$S\subset  [T,T].$
\end{definition}

The following properties hold for stem extensions of Lie superalgebras, which are analogous to the ones in Lie algebra case (see \cite{Batten-Stitzinger}, for example). Hereafter, we use a partial order in $\mathbb{Z} \times \mathbb{Z}$ as follows:
\begin{equation}\label{po123}
(m, n)\leq(k, l) \Longleftrightarrow m\leq k, n\leq l.
\end{equation}
For $m,n\in \mathbb{Z}$, we write $|(m,n)|=m+n.$ We also view $\mathbb{Z} \times \mathbb{Z}$ as the additive group in the usual way.

\begin{lemma}\label{lem6688} Let $L$ be a finite-dimensional Lie superalgebra. Suppose that
$$0\longrightarrow S\longrightarrow T\stackrel{\gamma}{\longrightarrow} L\longrightarrow 0$$
is a stem extension of $L$. Then the following statements hold.
\begin{itemize}
\item[(1)] Suppose that  $\mathrm{sdim}L=(s,t)$. Then both $S$ and $T$ are finite-dimensional and
$$\mathrm{sdim}S\leq \left(\frac{1}{2}s(s-1)+\frac{1}{2}t(t+1)+s,st\right).$$
\item[(2)] $S$ is contained in every maximal subalgebra of $T$.
\item [(3)] Suppose that  $
0\longrightarrow A{\longrightarrow} B \stackrel{\beta}{\longrightarrow} L \longrightarrow 0
$
is an extension of $L$ and  $f$ is a homomorphism of extensions:
$$
\xymatrix{0\ar[r]&A\ar[d]_{f}\ar[r]^{}&B\ar[d]_{f}\ar[r]^{\beta}&L\ar[d]_{\mathrm{id}}\ar[r]&0\\
0\ar[r]&S\ar[r]^{}&T\ar[r]^{\gamma}&L\ar[r]&0}
$$
Then
$f$ must be surjective.
\end{itemize}
\end{lemma}
\begin{proof}

(1) See \cite[Lemma 2.3]{zhang-liu}.

(2)   Let $X$  be a maximal subalgebra of $T$. Assume conversely that $S\not\subset X$. Since $S\subset\mathrm{Z}(T)$, one sees that $S+X=T$  and then $S\subset [T,T]\subset X$, a contradiction.

(3)  Assume conversely  that $f(B)\not=T$. By (1), $T$ is finite-dimensional. Then there is a maximal subalgebra $X$ such that $f(B)\subset X$. Note that the commutativity of the diagram implies  $T=S+f(B)$. Then by (2), we have $T\subset X$, a contradiction.
\end{proof}

We   note that an arbitrary extension
\begin{equation}\label{eqstemext342sf}
0\longrightarrow A{\longrightarrow} B \stackrel{\beta}{\longrightarrow} L \longrightarrow 0
\end{equation}
 induces naturally a central extension and then a stem extension of $L$. First, by factoring out the ideal $[A, B]$, we obtain a central extension
\begin{equation}\label{eqanyextcentral}
0\longrightarrow A/[A, B]\longrightarrow B/[A, B]  \stackrel{\beta}{\longrightarrow} L \longrightarrow 0
\end{equation}
 Then, consider any supplementary subsuperspace of $A\cap [B, B]/[A, B]$ in $A/[A, B]$, which must be of form $X/[A, B]$, where $X$ is a subsuperspace of $A$ containing $[A, B]$. Clearly, $X$ is an ideal of $A$ as well as $B$.
Then we have a decomposition of ideals:
\begin{equation}\label{eq-sum-suppl}
A/[A, B]=A\cap [B, B]/[A, B]\oplus X/[A, B].
\end{equation}
Such a subsuperspace (ideal) $X$ is called a \textit{stem denominator} of the extension (\ref{eqstemext342sf}). Since $X\subset A$ and $\beta(A)=0$,  we obtain an extension
 \begin{equation}\label{eqanyextcentralstemdeformation}
0\longrightarrow A/X\longrightarrow B/X  \stackrel{\beta}{\longrightarrow} L \longrightarrow 0,
\end{equation}
 called a \textit{stem deformation} of  extension (\ref{eqstemext342sf}). By the following proposition, a stem deformation of an extension is a stem extension.

 \begin{proposition}\label{pro-stem-defomation}
 Let
 \begin{equation}\label{eqanyextnew}
0\longrightarrow A\longrightarrow B  \stackrel{\beta}{\longrightarrow} L \longrightarrow 0
\end{equation}
be an extension of Lie superalgebra  $L$. Suppose that  $X$ is a subsuperspace of $A$ and $X\supset [A, B]$.
\begin{itemize}
\item[(1)] $X$ is a stem denominator of  extension (\ref{eqanyextnew}) if and only if
\begin{equation}\label{equiv-denomiator}
\mbox{$A=A\cap [B, B]+ X$  and  $ [B, B]\cap X=[A, B]$.}
\end{equation}

\item[(2)] Suppose that  $X$ is a stem denominator of  extension (\ref{eqanyextnew}). Then
 \begin{equation}\label{eqanyextcentralstemdeformation11}
0\longrightarrow A/X\longrightarrow B/X  \stackrel{\beta}{\longrightarrow} L \longrightarrow 0
\end{equation}
is a stem extension and moreover,
\begin{equation}\label{eq-sum-suppls-iso}
A/X\cong A\cap [B, B]/[A, B].
\end{equation}
\item[(3)] For an extension, a stem denominators and then a deformations always exist.
\end{itemize}
 \end{proposition}
\begin{proof} (1) It follows from the fact that (\ref{equiv-denomiator}) is equivalent to (\ref{eq-sum-suppl}).

(2) Since $[A,B]\subset X$, we have $[A/X, B/X]=0$. By (\ref{equiv-denomiator}), it is clear that $A/X\subset [B/X, B/X]$.  Hence (\ref{eqanyextcentralstemdeformation11}) is a stem extension. While
(\ref{eq-sum-suppls-iso}) is a direct consequence of (\ref{equiv-denomiator}).

(3) Form the argument before this proposition, one sees a stem denominator always exists and so does a stem deformation by (2).
\end{proof}

The following theorem tells us that if a stem extension $(S)$ of a Lie superalgebra $L$ is a homomorphic image of an extension $(E)$ of $L$, then $(S)$  must be a homomorphic image of some stem deformation of $(E)$.

\begin{theorem}\label{them-hom-of-stemk}  Let
\begin{equation}\label{usual-ext052}
0\longrightarrow A{\longrightarrow} B \stackrel{\beta}{\longrightarrow} L \longrightarrow 0
\end{equation}
be an extension  of $L$ and
\begin{equation} \label{eq-stemnnn}
0\longrightarrow S\longrightarrow T\stackrel{\gamma}{\longrightarrow} L\longrightarrow 0
\end{equation}
 a stem extension of $L$.  Suppose that  $f: B\longrightarrow T$ is a homomorphism of extensions:
\begin{equation}\label{eqtbjiahx213}
\xymatrix{0\ar[r]&A\ar[d]_{f}\ar[r]^{}&B\ar[d]_{f}\ar[r]^{\beta}&L\ar[d]_{\mathrm{id}}\ar[r]&0\\
0\ar[r]&S\ar[r]^{}&T\ar[r]^{\gamma}&L\ar[r]&0}
\end{equation}
Then there is a stem denominator $X$ of  extension (\ref{usual-ext052}) such that $f(X)=0$ and $f$ induces an epimorphism of extensions:
$$
\xymatrix{0\ar[r]&A/X\ar[d]_{f}\ar[r]^{}&B/X\ar[d]_{f}\ar[r]^{\beta}&L\ar[d]_{\mathrm{id}}\ar[r]&0\\
0\ar[r]&S\ar[r]^{}&T\ar[r]^{\gamma}&L\ar[r]&0}
$$
\end{theorem}
\begin{proof}
By Lemma \ref{lem6688}(3), $f$ is surjective. Then $f(A)=S$ and $f^{-1}(S)=A$. Since $S\subset [T, T]$, we have $S\subset f([B, B])$ and $S\subset f(A\cap [B, B])$. Hence $S=f(A\cap [B, B])$ and then $f(A)=f(A\cap [B, B])$.  Consequently,
\begin{equation}\label{eqkerbu}
A=A\cap [B, B]+\ker f.
\end{equation}
Clearly, $[A, B]\subset\ker f,$ since $S\subset \mathrm{Z}(T)$. By (\ref{eqkerbu}), there is a subsuperspace $X\subset \ker f$ satisfying that
$[A, B]\subset X\subset A$ such that $A=A\cap [B,B]+X$ and
 $[B,B]\cap X=[A,B]$. By Proposition \ref{pro-stem-defomation}, $X$ is the desired stem denominator.
The proof is complete.
\end{proof}

We shall prove an important fact, which states  that  a stem extension of $L$ must be  a homomorphic image  of a stem deformation of any free presentation of $L.$
Recall that  a Lie superalgebra $L$ always has  a free presentation, that is, an extension
\begin{equation*}
0\longrightarrow R\longrightarrow F\stackrel{\pi}{\longrightarrow} L\longrightarrow 0
\end{equation*}
with $F$ being a free Lie superalgebra. Let
\begin{equation} \label{eq-stemnnn104}
0\longrightarrow S\longrightarrow T\stackrel{\gamma}{\longrightarrow} L\longrightarrow 0
\end{equation}
 a stem extension of $L$.  Then there is  a  homomorphism of extensions,
\begin{equation}\label{eqtbjiahx}
\xymatrix{0\ar[r]&R\ar[d]_{f}\ar[r]^{}&F\ar[d]_{f}\ar[r]^{\pi}&L\ar[d]_{\mathrm{id}}\ar[r]&0\\
0\ar[r]&S\ar[r]^{}&T\ar[r]^{\gamma}&L\ar[r]&0}
\end{equation}
So, as a direct consequence of Theorem \ref{them-hom-of-stemk}, we have
\begin{theorem}\label{thm-free-stem-hom} Suppose that
 \begin{equation}\label{eq-free-pre1}
0\longrightarrow R\longrightarrow F\stackrel{\pi}{\longrightarrow} L\longrightarrow 0
\end{equation}
is a free presentation of  Lie superalgebra $L$ and
\begin{equation}\label{eq-stem-ext-any}
0\longrightarrow S\longrightarrow T\stackrel{\gamma}{\longrightarrow} L\longrightarrow 0
\end{equation}
is   a stem extension of $L$. Then  there is a stem denominator $X$ of (\ref{eq-free-pre1}) such that $f(X)=0$  and $f$ induces an epimorphism of extensions:
$$
\xymatrix{0\ar[r]&R/X\ar[d]_{f}\ar[r]^{}&F/X\ar[d]_{f}\ar[r]^{\pi}&L\ar[d]_{\mathrm{id}}\ar[r]&0\\
0\ar[r]&S\ar[r]^{}&T\ar[r]^{\gamma}&L\ar[r]&0}
$$\qed
\end{theorem}

In view of Lemma \ref{lem6688}(1),   the following notion makes sense.

\begin{definition}
Let $L$ be a finite-dimensional Lie superalgebra. A stem extension of $L$,
\begin{equation*}
0\longrightarrow S\longrightarrow T\stackrel{\gamma}{\longrightarrow} L\longrightarrow 0,
\end{equation*}
 is called \textit{maximal}, if,  among all the stem extensions of $L$, the kernel $S$ is of maximal superdimension with respect to the partial order (\ref{po123}).
\end{definition}

\begin{theorem}\label{thm-max-stem-ext} Let $L$ be a finite-dimensional Lie superalgebra and
\begin{equation}\label{frethemeqllls}
0\longrightarrow R\longrightarrow F\stackrel{\pi}{\longrightarrow} L\longrightarrow 0.
\end{equation}
a free presentation. Then,  up to isomorphisms of extensions, the stem deformations of (\ref{frethemeqllls}) exhaust all the maximal stem extensions of $L$.
\end{theorem}
\begin{proof} Let
\begin{equation}\label{eq-stem-ext-any-liu}
0\longrightarrow M\longrightarrow C\stackrel{\gamma}{\longrightarrow} L\longrightarrow 0
\end{equation}
   be a maximal stem extension of $L$. By Theorem \ref{thm-free-stem-hom}, there is  an epimorphism $f$ of extensions:
$$
\xymatrix{0\ar[r]&R/X\ar[d]_{f}\ar[r]^{}&F/X\ar[d]_{f}\ar[r]^{\pi}&L\ar[d]_{\mathrm{id}}\ar[r]&0
\\
0\ar[r]&M\ar[r]^{}&C\ar[r]^{\gamma}&L\ar[r]&0}
$$
where $X$ is a stem denominator of (\ref{frethemeqllls}) such that $f(X)=0$. By the maximality  of (\ref{eq-stem-ext-any-liu}), $f$ is an isomorphism of extensions.
By Proposition \ref{pro-stem-defomation}(2), $M\cong R\cap [F,F]/[R, F]$ is independent of the choice of $X$. Hence every stem deformation of (\ref{frethemeqllls}) is a maximal stem extension of $L$.
\end{proof}

\section{Multipliers and Covers}

Analogous to the Lie algebra case, we  introduce the following definition for    Lie superalgebras.
\begin{definition}\label{def-multiplier-cover} Let $L$ be a finite-dimensional Lie superalgebra. If
\begin{equation*}
0\longrightarrow M\longrightarrow C{\longrightarrow} L\longrightarrow 0
\end{equation*}
is a maximal stem extension of Lie superalgebra $L$, then $M$ is called a multiplier and $C$ a cover of $L.$
\end{definition}
By Theorem \ref{thm-max-stem-ext}, for a finite-dimensional Lie superalgebra, multipliers and covers always exist.
We shall prove that both multipliers and covers are unique up to Lie superalgebra isomorphisms. We need a technical lemma, which is a super-version of a general result (see \cite{Batten-Stitzinger}). Recall that a superalgebra is a superspace with a bilinear multiplication which is compatible with the $\mathbb{Z}_{2}$-grading structure.
 \begin{lemma}\label{lem1234}
   Let $A_{1}, A_{2}, B_{1}$ and $B_{2}$ be   superalgebras and
   \begin{equation}\label{eq091}
   A_{1}\oplus B_{1}\cong A_{2}\oplus B_{2}
 \end{equation}
 as superalgebras.
   Suppose that  $A_{1}$ and $A_{2}$ are finite-dimensional and
   $A_{1}\cong A_{2}$ as superalgebras.
Then $B_{1}\cong B_{2}$ as superalgebras.
\end{lemma}
   \begin{proof}
   By (\ref{eq091}), one can view $A_{2}$ and $B_{2}$ as $\mathbb{Z}_{2}$-graded ideals of the  superalgebra $A_{1}\oplus B_{1}$ and then we have
   \begin{equation}\label{eq093}
A_{1}\oplus B_{1}=A_{2}\oplus B_{2}.
\end{equation}
 If $A_{1}\cap B_{2}=0$ or $A_{2}\cap B_{1}=0$, one may easily see that $B_{1}\cong B_{2}$ as superalgebras, since $A_{1}$ and $A_{2}$ are finite-dimensional.

Suppose that  $A_{1}\cap B_{2}\neq 0$ and $A_{2}\cap B_{1}\neq 0$. Note that both $A_{1}\cap B_{2}$ and $A_{2}\cap B_{1}$ are $\mathbb{Z}_{2}$-graded ideals. By (\ref{eq093}), we have the following superalgebra isomorphism:
\begin{equation}\label{eq094}
\frac{A_{1}}{A_{1}\cap B_{2}}\oplus \frac{B_{1}}{A_{2}\cap B_{1}}\cong \frac{A_{2}}{A_{2}\cap B_{1}}\oplus \frac{B_{2}}{A_{1}\cap B_{2}}.
\end{equation}
Then we have the following superalgebra isomorphisms:
\begin{eqnarray*}
\frac{A_{1}}{A_{1}\cap B_{2}}\oplus \frac{A_{2}}{A_{2}\cap B_{1}}\oplus B_{1}&\cong& \frac{A_{2}}{A_{2}\cap B_{1}}\oplus \frac{A_{1}\oplus B_{1}}{A_{1}\cap B_{2}}\\&\cong& \frac{A_{2}}{A_{2}\cap B_{1}}\oplus \frac{A_{2}\oplus B_{2}}{A_{1}\cap B_{2}}\\&\cong& \frac{A_{2}}{A_{2}\cap B_{1}}\oplus \frac{B_{2}}{A_{1}\cap B_{2}}\oplus A_{2}.
 \end{eqnarray*}
By symmetry, since   $A_{1}\cong A_{2}$ as superalgebras,  it follows from  (\ref{eq094}) that the following superalgebra isomorphism holds:
$$\frac{A_{1}}{A_{1}\cap B_{2}}\oplus \frac{A_{2}}{A_{2}\cap B_{1}}\oplus B_{1}\cong \frac{A_{1}}{A_{1}\cap B_{2}}\oplus \frac{A_{2}}{A_{2}\cap B_{2}}\oplus B_{2}.$$
By induction on  dimensions, we have $B_{1}\cong B_{2}$ as superalgebras.
\end{proof}

\begin{theorem}\label{multiplier-unique-free} Let $L$ be a finite-dimensional Lie superalgebra. Up to Lie superalgebra isomorphisms, there are a unique multiplier and a unique cover of $L$, denoted by $\mathcal{M}(L)$ and $\mathcal{C}(L)$, respectively. Moreover, for any free presentation of $L$,
\begin{equation}\label{frethemeq54svls}
0\longrightarrow R\longrightarrow F{\longrightarrow} L\longrightarrow 0,
\end{equation}
\begin{itemize}
\item[(1)]
$
\mathcal{M}(L)\cong R\cap[F,F]/[R,R], $
\item[(2)]
 $\mathcal{C}(L)\cong Y/[R,F],$ where $Y$  is any subsuperspace $Y$ of $F$ containing $[F,F]$ such that $
F/[R, F]=X/[R, F]\oplus Y/[R, F]
$ as superspaces.
\end{itemize}
\end{theorem}
\begin{proof}
Let $X$ be a stem denominator of  (\ref{frethemeq54svls}). By Theorem \ref{thm-max-stem-ext}, it suffices to show that both $R/X$ and $F/X$ are independent of the choice of stem denominator  $X$. Moreover, by Theorem \ref{thm-max-stem-ext} again, one may assume
that $F$ is a free Lie superalgebra generated by a finite homogeneous set, since $L$ is finite-dimensional. Since $X$ is a stem denominator, we have
$[F,F]\cap X=[R, F]$ and then $[F,F]/[R, F]\cap X/[R, F]=0$. Then there is a subsuperspace $Y$ of $F$ containing $[F,F]$ such that we have a direct sum decomposition of superspaces:
\begin{equation}\label{eqyighesliu}
F/[R, F]=X/[R, F]\oplus Y/[R, F].
\end{equation}
Since $Y\supset [F,F]$, we have $Y\lhd F$. Hence
(\ref{eqyighesliu}) is also a direct sum decomposition of ideals.
Since $X$ is a stem denominator, we also have a direct sum decomposition of ideals:
$$
R/[R, F]=R\cap [F,F]/[R, F]\oplus X/[R, F].
$$
Then, up to Lie superalgebra isomorphisms, $X/[R, F]$ is  independent of the choice of $X$. Moreover,
$$
X/[R, F]\cong R/R\cap[F,F]\cong (R+[F,F])/[F,F]\subset F/[F,F].
$$
is  finite-dimensional. Then by Lemma \ref{lem1234}, it follows from (\ref{eqyighesliu}) that, up to Lie superalgebra isomorphisms,  $F/X\cong Y/[R,F]$ is independent of the choice of $X$.
\end{proof}

 \section{Heisenberg superalgebras}

In this section, we compute multipliers, covers and maximal stem extensions for Heisenberg  superalgebras of odd center.
Recall that a finite-dimensional Lie superalgebra $\mathfrak{g}$ is called a Heisenberg (Lie) superalgebra provided that $\mathfrak{g}^{2}=\mathrm{Z}(\mathfrak{g})$ and $\dim\mathrm{Z}(\mathfrak{g})=1$. Heisenberg superalgebras consist of two types according to the parity of the ceter (see \cite{RSV}).
\begin{enumerate}
\item[(1)] A Heisenberg superalgebra of even center, denoted by $H(p, q)$ with $p+q\geq1$, has a homogeneous basis (called standard)
$$\{u_{1}, u_{2}, \ldots, u_{p}, v_{1}, v_{2}, \ldots, v_{p}, z \mid w_{1}, w_{2}, \ldots, w_{q}\},$$

where
$$
|u_{i}|=|v_{j}|=|z|=\bar{0},\ |w_{k}|=\bar{1}; \ 1\leq i \leq p,\ 1\leq j \leq p, \ 1\leq k \leq q,
$$
and the multiplication is given by
$$
[u_{i}, v_{i}]=-[v_{i}, u_{i}]=z,\ [w_{k}, w_{k}]=z
$$
and the other brackets of basis elements vanishing.

\item[(2)] A Heisenberg superalgebra of odd center, denoted by $H(n)$ with $n\geq 1$, has a homogeneous basis (called standard)
$$\{u_{1}, u_{2}, \ldots, u_{n} \mid z, w_{1}, w_{2}, \ldots, w_{n}\},$$

where $$|u_{i}|=\bar{0},\ |w_{j}|=|z|=\bar{1};\ 1\leq i \leq m,\ 1\leq j \leq n,$$
and the multiplication is given by
$$
[u_{i}, w_{i}]=-[w_{i}, u_{i}]=z
$$
and the other brackets of basis elements vanishing.
\end{enumerate}

In \cite[Proposition 4.4]{zhang-liu} (see also \cite[Theorem 4.3]{S.Nayak}), the authors characterize the multipliers of  Heisenberg  superalgebras of even center:
\begin{equation*}
\mathrm{sdim}\mathcal{M}(H(p,q))=\left\{
               \begin{array}{ll}
                 (2p^{2}-p+\frac{1}{2}q^{2}+\frac{1}{2}q-1, 2pq), & \hbox{$p+q\geq2,$}\\
                 (0,0), & \hbox{$p=0,q=1,$}\\
                 (2,0), & \hbox{$p=1,q=0.$}
        \end{array}
       \right.
\end{equation*}

Let us give a maximal stem extension of the  Heisenberg   superalgebra  $H(n)$ of odd center.
 Suppose that
   \begin{equation*}
0\longrightarrow W\longrightarrow K{\longrightarrow} H(n)\longrightarrow 0
\end{equation*}
is a   stem extension. Then $W\subset K^{2}\cap \mathrm{Z}(K)$ and
$K/W\cong H(n)$.
Then $K/W$ has a standard basis
$$(a_{1}+W, \ldots, a_{n}+W \mid c+W, b_{1}+W, \ldots, b_{n}+W),$$
 where $a_{i},  b_{i}, c\in K$ with $|a_{i}|=\overline{0}$, $|b_{i}|=|c|=\overline{1}$. So one may assume that
   \begin{alignat*}{2}
     &[a_{i}, a_{j}]=y_{i,j}, && \quad 1\leq i < j \leq n,\\
    &[a_{i}, b_{i}]=c+y_{i}, && \quad 1\leq i \leq n,\\
    &[a_{i}, b_{j}]=z_{i,j}, && \quad 1\leq i \neq j \leq n,\\
    &[a_{i}, c]=m_{i}, && \quad 1\leq i \leq n,\\
   &[b_{i}, b_{j}]=w_{i,j}, && \quad 1\leq i \leq j \leq n,\\
     &[b_{i}, c]=n_{i}, && \quad 1\leq i \leq n,\\
    &[c, c]=t,
   \end{alignat*}
  where $y_{i,j},  y_{i},z_{i,j}, m_{i},w_{i,j}, n_{i}, t\in W$ and $|y_{i,j}|=|n_{i}|=|w_{i,j}|=|t|=\overline{0}$, $|y_{i}|=|z_{i,j}|=|m_{i}|=\overline{1}$. Furthermore, without loss of generality, one may assume that  $y_{1}=0$.
  Note that \begin{eqnarray*}
  &&n_i=[b_i,[a_i,b_i]]=\frac{1}{2}[a_{i},[b_{i}, b_{i}]]=0,
  \\
  &&t=[c,[a_{1},b_{1}]]=0.
  \end{eqnarray*}
 Then we rewrite
 \begin{alignat*}{2}
     &[a_{i}, a_{j}]=y_{i,j}, && \quad 1\leq i < j \leq n,\\
     &[a_{1}, b_{1}]=c,\\
    &[a_{i}, b_{i}]=c+y_{i}, && \quad 2\leq i \leq n,\\
    &[a_{i}, b_{j}]=z_{i,j}, && \quad 1\leq i \neq j \leq n,\\
    &[a_{i}, c]=m_{i}, && \quad 1\leq i \leq n,\\
   &[b_{i}, b_{j}]=w_{i,j}, && \quad 1\leq i \leq j \leq n,\\
     &[b_{i}, c]=0, && \quad 1\leq i \leq n,\\
    &[c, c]=0.
   \end{alignat*}
  Now it is obvious that  $W$ is spanned by the following elements:
   \begin{alignat*}{2}
    &y_{i,j}, &&\quad 1\leq i < j \leq n,\\
   &y_{i}, &&\quad 2\leq i \leq n,\\
   &z_{i,j}, &&\quad 1\leq i \neq j \leq n,\\
  & m_{i}, &&\quad 1\leq i \leq n,\\
  & w_{i,j}, &&\quad 1\leq i \leq j \leq n.
   \end{alignat*}
  Then $K$ is spanned by the  elements displayed above and $a_{1}, \ldots, a_{n}, b_{1}, \ldots, b_{n}$ and $c$.

 \textit{Case 1}: $n=1$. Suppressing all the subscripts, we have
   \begin{eqnarray*}
   [a, b]=c, \quad
   [a, c]=m,\quad
  [b, b]=w,\quad
  [b, c]=[c, c]=0.
   \end{eqnarray*}
   Then $W$ is generated by $w, m$.
    Hence $\mathrm{sdim} W\leq (1\mid 1)$.

    Now let $\widehat{H}(1)$ be a superspace with a basis  $(\widehat{a}, \widehat{w}\mid \widehat{b}, \widehat{c}, \widehat{m})$. Then  $\widehat{H}(1)$ becomes a Lie superalgebra by letting
     \begin{eqnarray*}
  [\widehat{a}, \widehat{b}]=-[\widehat{b}, \widehat{a}]=\widehat{c},\quad
 [\widehat{a}, \widehat{c}]=-[\widehat{c}, \widehat{a}]=\widehat{m}, \quad
   [\widehat{b}, \widehat{b}]=\widehat{w}
   \end{eqnarray*}
    and the other brackets of basis elements vanish.
       Let $\widehat{MH}(1)$ be the subsuperspace spanned by $\widehat{w}, \widehat{m}$. Then $\widehat{MH}(1)\subset \widehat{H}(1)^{2}\cap \mathrm{Z}(\widehat{H}(1))$ and $\widehat{H}(1)/\widehat{MH}(1)\cong H(1).$
        Since  $\mathrm{sdim} \widehat{MH}(1)= (1\mid 1)$, one sees that
        $$
        0\longrightarrow \widehat{MH}(1)\longrightarrow \widehat{H}(1){\longrightarrow} H(1)\longrightarrow 0
        $$
        is a maximal stem extension of $H(1)$.
        In particular, $\widehat{MH}(1)$ is a multiplier and $\widehat{H}(1)$ a cover of $H(1)$.

  \textit{Case 2}: $n\geq 2$. Consider any $i$ and take $j\neq i$. One may check that
    \begin{equation*}
    m_{i}=[a_{i}, [a_{j}, b_{j}]]=0.
   \end{equation*}
    Then  $W$ is spanned by the following elements:
   \begin{alignat*}{2}
   &y_{i,j}, && \quad 1\leq i < j \leq n,\\
   &y_{i}, && \quad 2\leq i \leq n,\\
   &z_{i,j}, && \quad 1\leq i \neq j \leq n,\\
   &w_{i,j}, && \quad 1\leq i \leq j \leq n.
  \end{alignat*}
  Hence $\mathrm{sdim} W\leq (n^{2}\mid n^{2}-1)$.

    Now let $\widehat{H}(n)$  with $n\geq 2$ be a superspace with a basis consisting of  even elements
\begin{alignat*}{2}
  &\widehat{a}_{i}, && \quad 1\leq i\leq n,\\
  &\widehat{y}_{i, j}, && \quad 1\leq i<j\leq n,\\
  &\widehat{w}_{i, j}, && \quad 1\leq i\leq j\leq n,
  \end{alignat*}
   and odd  elements
    \begin{alignat*}{2}
    &\widehat{b}_{i}, && \quad 1\leq i\leq n,
     \\
     &\widehat{c},\\
      &\widehat{y}_{i}, && \quad 2\leq i \leq n,\\
       &\widehat{z}_{i, j}, && \quad 1\leq i \neq j \leq n.
       \end{alignat*}
      Then one may check that $\widehat{H}(n)$ becomes a Lie superalgebra  by letting
\begin{alignat*}{2}
&[\widehat{a}_{i}, \widehat{a_{j}}]=-[\widehat{a_{j}}, \widehat{a}_{i}]=\widehat{y}_{i,j}, && \quad 1\leq i < j \leq n,\\
&[\widehat{a}_{1}, \widehat{b}_{1}]=-[\widehat{b}_{1}, \widehat{a}_{1}]=\widehat{c},\\
&[\widehat{a}_{i}, \widehat{b}_{i}]=-[\widehat{b}_{i}, \widehat{a}_{i}]=\widehat{c}+\widehat{y}_{i}, && \quad  2\leq i \leq n,\\
&[\widehat{a}_{i}, \widehat{b}_{j}]=-[\widehat{b}_{j},\widehat{a}_{i}]=\widehat{z}_{i,j}, && \quad 1\leq i \neq j \leq n,\\
&[\widehat{b}_{i}, \widehat{b}_{j}]=-[\widehat{b}_{j}, \widehat{b}_{i}]=\widehat{w}_{i,j}, && \quad 1\leq i < j \leq n,
  \end{alignat*}
and the other brackets of basis elements vanish.
       Let $\widehat{MH}(n)$ be the subsuperspace spanned by $\widehat{y}_{i, j}, \widehat{w}_{i, j},\widehat{y}_{i}, \widehat{z}_{i, j} $. Then $\widehat{MH}(n)\subset \widehat{H}(n)^{2}\cap \mathrm{Z}(\widehat{H}(n))$ and $\widehat{H}(n)/\widehat{MH}(n)\cong H(n).$
       Since  $\mathrm{sdim} \widehat{MH}(n)= (n^{2}\mid n^{2}-1)$, one sees that
       $$
        0\longrightarrow \widehat{MH}(n)\longrightarrow \widehat{H}(n){\longrightarrow} H(n)\longrightarrow 0
        $$
        is a maximal stem extension of. In particular, $\widehat{MH}(n)$ is a multiplier and $\widehat{H}(n)$ a cover of $H(n)$.

Summarizing, we have
\begin{theorem} Let $n$ be a positive integer. Then
$$
0\longrightarrow \widehat{MH}(n)\longrightarrow \widehat{H}(n){\longrightarrow} H(n)\longrightarrow 0
$$
is a maximal stem extension of the Heisenberg superalgebra $H(n)$ of odd center. In particular,
$\widehat{H}(n)$ is the cover of $H(n)$ and
\begin{equation*}
    \mathrm{sdim}\mathcal{M}(H(n))=\left\{
               \begin{array}{ll}
                 (n^{2}\mid n^{2}-1), & \hbox{$n\geq2$}\\
                 (1\mid 1), & \hbox{$n=1.$}\\
        \end{array}
       \right.
   \end{equation*}
\end{theorem}

\section{Model filiform superalgebras}

Suppose that  $n$ is a positive integer and $m$ a nonnegative integer. Let $F(n,m)$ be a Lie superalgebra with  basis
\begin{equation}\label{q2}
\{x_{0}, \ldots, x_{n}\mid y_{1}, \ldots, y_{m}\},
\end{equation}
where
$$
|x_{i}|=\bar{0},\ |y_{j}|=\bar{1}; \ 0\leq i\leq n , \ 1\leq j\leq m,
$$
and  multiplication given by
$$
[x_{0}, x_{i}]=x_{i+1},\ [x_{0}, y_{j}]=y_{j+1}
$$
and the other brackets of basis elements vanishing.
It is easy to see that $F(n,m)$ is a nilpotent Lie superalgebra of super-nilindex $(n,m)$ (see\cite{yang-liu}, for example).  Note that $F(1,0)$ is an abelian Lie algebra and   $ F(1,1)$ are an abelian Lie superalgebra. Note that $F(n,0)$ with $n>1$ is just the model filiform Lie algebra of nilindex $n$ (see \cite{A-F}, for example). We call $F(n,m)$ with $(n,m)\neq (1,0), (1,1)$ the \textit{model filiform Lie superalgebra} of super-nilindex $(n,m)$. For further information on filiform Lie (super)algebras, the reader is referred to \cite{A-F,M-G}, for example.

   Let us give a maximal stem extension of the  model filiform Lie superalgebra  $F(n,m)$.
 Suppose that
   \begin{equation*}
0\longrightarrow W\longrightarrow K{\longrightarrow} F(n,m)\longrightarrow 0
\end{equation*}
is a   stem extension. Then $W\subset K^{2}\cap \mathrm{Z}(K)$ and
$K/W\cong F(n,m)$.
Then $K/W$ has a standard basis
$$(a_{0}+W, a_{1}+W, \ldots, a_{n}+W \mid b_{1}+W, \ldots, b_{m}+W),$$
 where $a_{i},  b_{j}\in K$ with $|a_{i}|=\overline{0}$, $|b_{j}|=\overline{1}$. Then we have
  \begin{alignat*}{2}
   &[a_{0}, a_{i}]=a_{i+1}+x_{i},&& \quad 1\leq i \leq n-1,\\
   &[a_{i}, a_{j}]=y_{i,j},&& \quad 1\leq i < j \leq n,\\
   &[a_{0}, b_{j}]=b_{j+1}+y_{j},&& \quad 1\leq j \leq m-1,\\
   &[b_{i}, b_{j}]=z_{i,j},&& \quad 1\leq i \leq j \leq m,\\
   &[a_{i}, b_{j}]=t_{i,j},&& \quad 1\leq i \leq n, 1\leq j \leq m,
   \end{alignat*}
  where $x_{i}, y_{i,j}, y_{j}, z_{i,j}, t_{i,j}\in W$ and $|x_{i}|=|y_{i,j}|=|z_{i,j}|=\overline{0}$, $|y_{j}|=|t_{i,j}|=\overline{1}$. Since $W\subset \mathrm{Z}(K)$, without loss of generality, one may assume that $x_{i}=y_{j}=0$.
Moreover, using the super Jacobi identity one may check the following identities:
   \begin{itemize}
  \item[(1)] $y_{i,j+1}=-y_{i+1,j}$ for $1\leq i,j\leq n-1$.
  \item[(2)] $t_{i,j+1}=-t_{i+1,j}$ for $1\leq i\leq n-1$, $1\leq j\leq m-1$.
  \item[(3)] $t_{1,j+1}=0$ for $1\leq j\leq m-1$.
  \item[(4)] $t_{i+1,1}=0$ for $1\leq i\leq n-1$.
  \item[(5)] $z_{i,j+1}=-z_{i+1,j}$ for $1\leq i\neq j\leq m-1$.
  \item[(6)] $2z_{j,j+1}=0$ for $1\leq j\leq m-1$.

 \end{itemize}

  \textit{Case 1}: $n\geq 2$, $m=0$. In this case, $F(n,0)$ is a model filiform Lie algebra.
  By $(1)$,
  we rewrite
  \begin{alignat*}{2}
   &[a_{0}, a_{i}]=a_{i+1},&& \quad 1\leq i \leq n-1,\\
   &[a_{1}, a_{2}]=y_{1,2},\\
   &[a_{n-1}, a_{n}]=y_{n-1,n},\\
   &[a_{i}, a_{j}]=y_{i,j}=-y_{i+1,j-1}=-[a_{i+1}, a_{j-1}], && \quad 2<j-i\;\mbox{being odd}.
   \end{alignat*}
    Now it is obvious that  $W$ is spanned by the elements $y_{i,i+1}$, $1\leq i\leq n-1.$
 Hence $\mathrm{sdim} W\leq (n-1\mid 0)$. Then $K$ is spanned by the  elements displayed above and $a_{0}, a_{1}, \ldots, a_{n}$.

     Now let $\widehat{F}(n,0)$ with $n\geq 2$ be a superspace with a basis consisting of  even elements
\begin{alignat*}{2}
  &\widehat{a}_{i}, && \quad 0\leq i\leq n,\\
  &\widehat{y}_{j}, && \quad 2\leq j\leq n.
  \end{alignat*}
  Then one may check that $\widehat{F}(n,0)$ becomes a Lie superalgebra  by letting
\begin{alignat*}{2}
&[\widehat{a}_{0}, \widehat{a_{i}}]=-[\widehat{a_{i}}, \widehat{a}_{0}]=\widehat{a}_{i+1}, && \quad 1\leq i \leq n-1,\\
&[\widehat{a}_{1}, \widehat{a}_{2}]=-[\widehat{a}_{2}, \widehat{a}_{1}]=\widehat{y}_{2},\\
&[\widehat{a}_{n-1}, \widehat{a}_{n}]=-[\widehat{a}_{n}, \widehat{a}_{n-1}]=\widehat{y}_{n},\\
&[\widehat{a}_{i}, \widehat{a}_{j}]=-[\widehat{a}_{j}, \widehat{a}_{i}]=-[\widehat{a}_{i+1}, \widehat{a}_{j-1}]=[\widehat{a}_{j-1}, \widehat{a}_{i+1}]=\widehat{y}_{j-1}, && \quad 2<j-i\;\mbox{being odd},
\end{alignat*}
and the other brackets of basis elements vanish.
       Let $\widehat{MF}(n,0)$ be the subsuperspace spanned by all $\widehat{y}_{j} $. Then $\widehat{MF}(n,0)\subset \widehat{F}(n,0)^{2}\cap \mathrm{Z}(\widehat{F}(n,0))$ and $\widehat{F}(n,0)/\widehat{MF}(n,0)\cong F(n,0).$
       Since  $\mathrm{sdim} \widehat{MF}(n,0)= (n-1\mid 0)$, one sees that
       $$
        0\longrightarrow \widehat{MF}(n,0)\longrightarrow \widehat{F}(n,0){\longrightarrow} F(n,0)\longrightarrow 0
        $$
        is a maximal stem extension of $F(n,0)$. In particular, $\widehat{MF}(n,0)$ is a multiplier and $\widehat{F}(n,0)$ a cover of $F(n,0)$.

 \textit{Case 2}: $n\geq 2$, $m=1$.
 By $(1)$ and $(4)$,
  we rewrite
  \begin{alignat*}{2}
   &[a_{0}, a_{i}]=a_{i+1},&& \quad 1\leq i \leq n-1,\\
   &[a_{1}, a_{2}]=y_{1,2},\\
   &[a_{n-1}, a_{n}]=y_{n-1,n},\\
   &[a_{i}, a_{j}]=y_{i,j}=-y_{i+1,j-1}=-[a_{i+1}, a_{j-1}], && \quad 2<j-i\;\mbox{being odd},\\
   &[a_{1}, b_{1}]=t_{1,1},\\
   &[b_{1}, b_{1}]=z_{1,1}.
   \end{alignat*}
    Now it is obvious that  $W$ is spanned by the elements $t_{1,1}$, $z_{1,1}$, $y_{i,i+1}$, $1\leq i\leq n-1$.
    Hence $\mathrm{sdim} W\leq (n\mid 1)$. Then $K$ is spanned by the  elements displayed above and $a_{0}, a_{1}, \ldots, a_{n}$.

     Now let $\widehat{F}(n,1)$ with $n\geq 2$ be a superspace with a basis consisting of  even elements
 \begin{alignat*}{2}
  &\widehat{a}_{i}, && \quad 0\leq i\leq n,\\
  &\widehat{y}_{j}, && \quad 2\leq j\leq n,\\
   &\widehat{z},
   \end{alignat*}
  and odd  element $\widehat{t}$.
 Then one may check that $\widehat{F}(n,1)$ becomes a Lie superalgebra  by letting
\begin{alignat*}{2}
&[\widehat{a}_{0}, \widehat{a_{i}}]=-[\widehat{a_{i}}, \widehat{a}_{0}]=\widehat{a}_{i+1}, && \quad 1\leq i \leq n-1,\\
&[\widehat{a}_{1}, \widehat{a}_{2}]=-[\widehat{a}_{2}, \widehat{a}_{1}]=\widehat{y}_{2},\\
&[\widehat{a}_{n-1}, \widehat{a}_{n}]=-[\widehat{a}_{n}, \widehat{a}_{n-1}]=\widehat{y}_{n},\\
&[\widehat{a}_{i}, \widehat{a}_{j}]=-[\widehat{a}_{j}, \widehat{a}_{i}]=-[\widehat{a}_{i+1}, \widehat{a}_{j-1}]=[\widehat{a}_{j-1}, \widehat{a}_{i+1}]=\widehat{y}_{j-1}, && \quad 2<j-i\;\mbox{being odd},\\
&[\widehat{a}_{1}, \widehat{b_{1}}]=-[\widehat{b_{1}}, \widehat{a}_{1}]=\widehat{t},\\
&[\widehat{b}_{1}, \widehat{b_{1}}]=-[\widehat{b_{1}}, \widehat{b}_{1}]=\widehat{z},
\end{alignat*}
and the other brackets of basis elements vanish.
       Let $\widehat{MF}(n,1)$ be the subsuperspace spanned by $\widehat{t},\widehat{z}$ and all $\widehat{y}_{j}$. Then
       $\widehat{MF}(n,1)\subset \widehat{F}(n,1)^{2}\cap \mathrm{Z}(\widehat{F}(n,1))$ and $\widehat{F}(n,1)/\widehat{MF}(n,1)\cong F(n,1).$
       Since  $\mathrm{sdim} \widehat{MF}(n,1)= (n\mid 1)$, one sees that
       $$
        0\longrightarrow \widehat{MF}(n,1)\longrightarrow \widehat{F}(n,1){\longrightarrow} F(n,1)\longrightarrow 0
        $$
        is a maximal stem extension of $F(n,1)$. In particular, $\widehat{MF}(n,1)$ is a multiplier and $\widehat{F}(n,1)$ a cover of $F(n,1)$.

  \textit{Case 3}: $n\geq 2$, $m\geq 2$. By $(1),(2),(5)$ and $(6)$,
 we rewrite
  \begin{alignat*}{2}
   &[a_{0}, a_{i}]=a_{i+1},&& \quad 1\leq i \leq n-1,\\
   &[a_{1}, a_{2}]=y_{1,2},\\
   &[a_{n-1}, a_{n}]=y_{n-1,n},\\
   &[a_{i}, a_{j}]=y_{i,j}=-y_{i+1,j-1}=-[a_{i+1}, a_{j-1}], && \quad  2<j-i\;\mbox{being odd},\\
   &[a_{0}, b_{j}]=b_{j+1},&& \quad 1\leq j \leq m-1,\\
   &[a_{1}, b_{1}]=t_{1,1},\\
   &[a_{n}, b_{m}]=t_{n,m},\\
   &[a_{i}, b_{j}]=t_{i,j}=-t_{i+1,j-1}=-[a_{i+1}, b_{j-1}],&& \quad 1\leq i\leq n-1, 2\leq j \leq m.
   \end{alignat*}
    Now it is obvious that  $W$ is spanned by the following elements:
   \begin{alignat*}{2}
   &y_{i,i+1},&& \quad 1\leq i\leq n-1,\\
   &t_{i,1},&& \quad 1\leq i \leq n,\\
   &t_{n,j},&& \quad 2\leq j \leq m.
   \end{alignat*}
  Hence $\mathrm{sdim} W\leq (n-1\mid n+m-1)$. Then $K$ is spanned by the  elements displayed above and $a_{0}, a_{1}, \ldots, a_{n}, b_{1}, \ldots, b_{m}$.

     Now let $\widehat{F}(n,m)$ with $n,m\geq 2$ be a superspace with a basis consisting of  even elements
 \begin{alignat*}{2}
  &\widehat{a}_{i}, && \quad 0\leq i\leq n,\\
  &\widehat{y}_{j},&& \quad 2\leq j\leq n,
  \end{alignat*}
   and odd  elements
    \begin{alignat*}{2}
    &\widehat{b}_{p}, && \quad 1\leq p\leq m,
     \\
      &\widehat{t}_{q},&& \quad 2\leq q\leq m+n.
       \end{alignat*}
      Then one may check that $\widehat{F}(n,m)$ becomes a Lie superalgebra  by letting
\begin{alignat*}{2}
&[\widehat{a}_{0}, \widehat{a_{i}}]=-[\widehat{a_{i}}, \widehat{a}_{0}]=\widehat{a}_{i+1}, && \quad 1\leq i \leq n-1,\\
&[\widehat{a}_{1}, \widehat{a}_{2}]=-[\widehat{a}_{2}, \widehat{a}_{1}]=\widehat{y}_{2},\\
&[\widehat{a}_{n-1}, \widehat{a}_{n}]=-[\widehat{a}_{n}, \widehat{a}_{n-1}]=\widehat{y}_{n},\\
&[\widehat{a}_{i}, \widehat{a}_{j}]=-[\widehat{a}_{j}, \widehat{a}_{i}]=-[\widehat{a}_{i+1}, \widehat{a}_{j-1}]=[\widehat{a}_{j-1}, \widehat{a}_{i+1}]=\widehat{y}_{j-1}, && \quad 2<j-i\;\mbox{being odd},\\
&[\widehat{a}_{0}, \widehat{b}_{p}]=-[\widehat{b}_{p}, \widehat{a}_{0}]=\widehat{b}_{p+1}, && \quad 1\leq p \leq m-1,\\
&[\widehat{a}_{1}, \widehat{b}_{1}]=-[\widehat{b}_{1}, \widehat{a}_{1}]=\widehat{t}_{2},\\
&[\widehat{a}_{n}, \widehat{b}_{m}]=-[\widehat{b}_{m}, \widehat{a}_{n}]=\widehat{t}_{m+n},\\
&[\widehat{a}_{i}, \widehat{b}_{p}]=-[\widehat{b}_{p}, \widehat{a}_{i}]=-[\widehat{a}_{i+1}, \widehat{b}_{p-1}]=[\widehat{b}_{p-1}, \widehat{a}_{i+1}]=\widehat{t}_{i+p}, && \quad 1\leq i\leq n-1, 2\leq p \leq m.
  \end{alignat*}
and the other brackets of basis elements vanish.
       Let $\widehat{MF}(n,m)$ be the subsuperspace spanned by all $\widehat{y}_{j},\widehat{t}_{q}$. Then $\widehat{MF}(n,m)\subset \widehat{F}(n,m)^{2}\cap \mathrm{Z}(\widehat{F}(n,m))$ and $\widehat{F}(n,m)/\widehat{MF}(n,m)\cong F(n,m).$
       Since  $\mathrm{sdim} \widehat{MF}(n,m)= (n-1\mid m+n-1)$, one sees that
       $$
        0\longrightarrow \widehat{MF}(n,m)\longrightarrow \widehat{F}(n,m){\longrightarrow} F(n,m)\longrightarrow 0
        $$
        is a maximal stem extension of $F(n,m)$. In particular, $\widehat{MF}(n,m)$ is a multiplier and $\widehat{F}(n,m)$ a cover of $F(n,m)$.

   \textit{Case 4}: $n=1$, $m\geq 2$. By $(3),(5)$ and $(6)$,
  we rewrite
   \begin{alignat*}{2}
   &[a_{0}, b_{j}]=b_{j+1},&& \quad 1\leq j \leq m-1,\\
   &[a_{1}, b_{1}]=t_{1,1}.
   \end{alignat*}
   Now it is obvious that  $W$ is spanned by the element $t_{1,1}.$
   Hence $\mathrm{sdim} W\leq (0\mid 1)$. Then $K$ is spanned by the  elements displayed above and $a_{0}, a_{1}, b_{1}, \ldots, b_{m}$.

     Now let $\widehat{F}(1,m)$ with $m\geq 2$ be a superspace with a basis consisting of  even element $\widehat{a}_{0}, \widehat{a}_{1}$
   and odd  elements $\widehat{t}$, $\widehat{b}_{p}$, $1\leq p\leq m$.

Then one may check that $\widehat{F}(1,m)$ becomes a Lie superalgebra  by letting
\begin{alignat*}{2}
&[\widehat{a}_{0}, \widehat{b}_{p}]=-[\widehat{b}_{p}, \widehat{a}_{0}]=\widehat{b}_{p+1}, && \quad 1\leq p \leq m-1,\\
&[\widehat{a}_{1}, \widehat{b_{1}}]=-[\widehat{b_{1}}, \widehat{a}_{1}]=\widehat{t},
\end{alignat*}
and the other brackets of basis elements vanish.
       Let $\widehat{MF}(1,m)$ be the subsuperspace spanned by $\widehat{t}$. Then $\widehat{MF}(1,m)\subset \widehat{F}(1,m)^{2}\cap \mathrm{Z}(\widehat{F}(1,m))$ and $\widehat{F}(1,m)/\widehat{MF}(1,m)\cong F(1,m).$
       Since  $\mathrm{sdim} \widehat{MF}(1,m)= (0\mid 1)$, one sees that
       $$
        0\longrightarrow \widehat{MF}(1,m)\longrightarrow \widehat{F}(1,m){\longrightarrow} F(1,m)\longrightarrow 0
        $$
        is a maximal stem extension of $F(1,m)$. In particular, $\widehat{MF}(1,m)$ is a multiplier and $\widehat{F}(1,m)$ a cover of $F(1,m)$.

Summarizing, we have:
\begin{theorem} Let $n,m$ be positive integers. Then
$$
0\longrightarrow \widehat{MF}(n,m)\longrightarrow \widehat{F}(n,m){\longrightarrow} F(n,m)\longrightarrow 0
$$
is a maximal stem extension of the model filiform Lie superalgebra $F(n,m)$. In particular,
$\widehat{F}(n,m)$ is the cover of $F(n,m)$ and
\begin{equation*}
    \mathrm{sdim}\mathcal{M}(F(n,m))=\left\{
               \begin{array}{ll}
                 (n-1\mid 0), & \hbox{$n\geq2$, $m=0,$}\\
                 (n\mid 1), & \hbox{$n\geq2$, $m=1,$}\\
                 (n-1\mid n+m-1), & \hbox{$n\geq2$, $m\geq 2,$}\\
                 (0\mid 1), & \hbox{$n=1$, $m\geq2$}.\\
          \end{array}
       \right.
   \end{equation*}
\end{theorem}

\end{document}